\documentclass[twoside,11pt]{article}

\usepackage{jmlr2e}

\usepackage{lipsum}
\usepackage{amsfonts}
\usepackage{graphicx}
\usepackage{epstopdf}
\usepackage[ruled,linesnumbered]{algorithm2e}
\usepackage{algorithmic}
\usepackage{amsmath}
\usepackage{amssymb}
\usepackage{bm}
\usepackage{mathbbold}
\usepackage{mathrsfs}
\usepackage{multirow}
\usepackage{dcolumn}
\usepackage{listings}
\usepackage{subfigure}
\usepackage{enumerate}

\newcommand{\dataset}{{\cal D}}
\newcommand{\fracpartial}[2]{\frac{\partial #1}{\partial  #2}}

\newcommand{\ud}{\mathrm{d}}
\newtheorem{assumption}{Assumption~}

\ShortHeadings{Speed up the convergence rate of stochastic gradients?}{X Xu and X Luo}
\firstpageno{1}

\begin{document}

\title{Can speed up the convergence rate of stochastic gradient methods to $\mathcal{O}(1/k^2)$ by a gradient averaging strategy?}

%\title{An accelerated stochastic gradient method with convergence rate $O(1/k^2)$ for strongly convex stochastic optimization}

\author{\name Xin Xu \email xu.permenant@gmail.com\\
    \name Xiaopeng Luo \email luo.permenant@gmail.com\\
	\addr Department of Chemistry\\
	Princeton University\\
	Princeton, NJ 08544, USA\\
	\addr Department of Control and Systems Engineering\\
	School of Management and Engineering\\
    Nanjing University\\
	Nanjing, 210093, China
    %\AND
    %\name Liang Liu \email mf1915058@smail.nju.edu.cn\\
    %\addr Department of Control and Systems Engineering\\
	%School of Management and Engineering\\
    %Nanjing University\\
	%Nanjing, 210093, China
}

%arXiv
\editor{}

\maketitle

\begin{abstract}%   <- trailing '%' for backward compatibility of .sty file
 In this paper we consider the question of whether it is possible to apply a gradient averaging strategy to improve on the sublinear convergence rates without any increase in storage. Our analysis reveals that a positive answer requires an appropriate averaging strategy and iterations that satisfy the variance dominant condition. As an interesting fact, we show that if the iterative variance we defined is always dominant even a little bit in the stochastic gradient iterations, the proposed gradient averaging strategy can increase the convergence rate $\mathcal{O}(1/k)$ to $\mathcal{O}(1/k^2)$ in probability for the strongly convex objectives with Lipschitz gradients. This conclusion suggests how we should control the stochastic gradient iterations to improve the rate of convergence.
\end{abstract}


\section{Introduction}
\label{AM:s1}

In this paper we consider the question of whether it is possible to apply a gradient averaging strategy to improve on the sublinear convergence rates without any increase in storage for stochastic gradient (SG) methods. The SG method is the popular methodology \citep{ZinkevichM2003A_SG,ZhangT2004A_SG,BottouL2007A_TradeoffsLearning,
NemirovskiA2009A_StochasticProgramming,ShwartzS2011A_subgradientSVM} for solving the following class of stochastic optimization problems:
\begin{equation}\label{AM:eq:OP}
  x_*=\arg\min_{x\in\mathbb{R}^d}F(x),
\end{equation}
where the real-valued function $F$ is defined by
\begin{equation}\label{AM:eq:f}
  F(x):=\mathbb{E}_\xi\big[f(x,\xi)\big]=\int_\Xi f(x,\xi)\ud P(\xi),
\end{equation}
and $\{f(\cdot,\xi),\xi\in\Xi\}$ can be defined as a collection of real-valued functions with a certain probability distribution $P$ over the index set $\Xi$. 

With an initial point $x_1$, these methods are characterized by the iteration
\begin{equation}\label{AM:eq:GD}
  x_{k+1}=x_k-\alpha_kg(x_k,\xi_k),
\end{equation}
where $\alpha_k>0$ is the stepsize and $g(x_k,\xi_k)$ is the stochastic gradient defined by
\begin{equation}\label{AM:eq:s}
\begin{aligned}
  g(x_k,\xi_k)=\left\{
  \begin{array}{c}
    \nabla f(x_k,\xi_k), \\[0.5em]
    \frac{1}{n_k}\sum_{i=1}^{n_k}\nabla f(x_k,\xi_{k,i}),
  \end{array}\right.
\end{aligned}
\end{equation}
which is usually assumed to be an unbiased estimate of the socalled full gradient $\nabla F(x_k)$, i.e., $\mathbb{E}_{\xi_k}g(x_k,\xi_k)=\nabla F(x_k)$ \citep{ShapiroA2009M_StochasticProgramming,BottouL2018R_SGD}. The SG method was originally developed by \cite{RobbinsH1951A_SG} for smooth stochastic approximation problems. It is guaranteed to achieve the sublinear convergence rate $\mathcal{O}(1/k)$ for strongly convex objectives \citep{NemirovskiA2009A_StochasticProgramming,
BottouL2018R_SGD} and this theoretical rate is also supported by practical experience \citep{BottouL2018R_SGD}. In particular, the practical performance of stochastic gradient methods with momentum has made them popular in the community working on training DNNs \citep{SutskeverI2013_DNN&SGM}; in fact, the approach could be viewed as a gradient averaging strategy. While momentum can lead to improved practical performance, it is still not known to lead to a faster convergence rate. Usually, the gradient averaging strategy and its variants can improve the constants in the convergence rate \citep{XiaoL2010A_GradientAveraging}, it does not improve on the sublinear convergence rates for SG methods. However, owing to the successful practical performance of gradient averaging, it is worth considering whether it is possible to improve the convergence rate, which forms the starting point of this work.

The primary contribution of this work is to show that under the variance dominant condition (Assumption \ref{AM:ass:A3}), the proposed gradient averaging strategy could improve the convergence rate $\mathcal{O}(1/k)$ to $\mathcal{O}(1/k^2)$ in probability without any increase in storage for the strongly convex objectives with Lipschitz gradients. This result also suggests how we should control the stochastic gradient iterations to improve the rate of convergence in practice. In particular, our averaging strategy coordinates the relationship between the mean and variance of the increment of the iteration, so that the growth of expectation can be controlled when the variance is reduced.

\subsection{Related Work}
\label{AM:s11}

We briefly review several methods related to the new strategy, mainly including stochastic gradient with momentum (SGM), gradient averaging, stochastic variance reduced gradient (SVRG), SAGA and iterate averaging.

\textbf{SGM.} With an initial point $x_1$, two scalar sequences $\{\alpha_k\}$ and $\{\beta_k\}$ that are either predetermined or set dynamically, and $x_0:=x_1$, SGM uses iterations of the form \citep{TsengP1988A_SGmomentum,BottouL2018R_SGD}
\begin{equation*}
  x_{k+1}=x_k-\alpha_kg(x_k,\xi_k)+\beta_k(x_k-x_{k-1}).
\end{equation*}
They are procedures in which each step is chosen as a combination of the stochastic gradient direction and the most recent iterate displacement. It is common to set $\alpha_k=\alpha$ and $\beta_k=\beta$ as some fixed constants, and in this case we can rewrite the SGM iteration as
\begin{equation}\label{AM:eq:SGM}
  x_{k+1}=x_k-\alpha\sum_{i=1}^k\beta^{k-i}g(x_i,\xi_i)
  =x_k-\frac{\alpha'_k}{\sum_{i=1}^k\beta^{k-i}}
  \sum_{i=1}^k\beta^{k-i}g(x_i,\xi_i).
\end{equation}
So it is clear that SGM is a weighted average of all previous stochastic gradient directions. In deterministic settings, it is referred to as the heavy ball method \citep{PolyakB1964A_momentum}. While SGM can lead to improved practical performance, it is still not known to lead to a faster convergence rate. Moreover, see Remark \ref{AM:rem:SGM} in Section \ref{AM:s4} for a variance analysis of SGM.

\textbf{Gradient Averaging.} Similar to SGM, gradient averaging is using the average of all previous gradients,
\begin{equation}\label{AM:eq:GA}
  x_{k+1}=x_k-\frac{\alpha_k}{k}\sum_{i=1}^kg(x_i,\xi_i).
\end{equation}
This approach is used in the dual averaging method \citep{NesterovY2009A_dualSG}. Compared with our new strategy, this method reduces the variance to a similar order $\mathcal{O}(1/k)$ without considering the stepsize $\alpha_k$, however, its expectation is not well controlled, for details see Remark \ref{AM:rem:GA} in Section \ref{AM:s4}. So as mentioned above, it can improve the constants in the convergence rate \citep{XiaoL2010A_GradientAveraging} but does not improve on the sublinear convergence rates.

\textbf{SVRG.} SVRG is designed to minimize the objective function of the form of a finite sum \citep{JohnsonR2013_SVRG}, i.e., 
\begin{equation*}
  F(x)=\frac{1}{n}\sum_{i=1}^nf_i(x).
\end{equation*}
The method is able to achieve a linear convergence rate for strongly convex problems, i.e.,
\begin{equation*}
  \mathbb{E}[F(x_{k+1})-F_*]\leqslant\rho\mathbb{E}[F(x_k)-F_*].
\end{equation*}
SVRG needs to compute the batch gradients $\nabla F(x_k)$ and has two parameters that needs to be set: besides the stepsize $\alpha$, there is an additional parameter $m$, namely the number of iterations per inner loop. In order to guarantee a linear convergence theoretically, one needs to choose $\alpha$ and $m$ such that
\begin{equation*}
  \rho:=\frac{1}{1-2\alpha L}
  \left(\frac{1}{m\alpha l}+2\alpha L\right)<1,
\end{equation*}
where $l$ and $L$ are given in Assumption \ref{AM:ass:A1}. Without explicit knowledge $l$ and $L$, the lengths of the inner loop $m$ and the stepsize $\alpha$ are typically both chosen by experimentation. 
This improved rate is achieved by either an increase in computation or an increase in storage. Hence, SVRG usually can not beat SG for very large $n$ \citep{BottouL2018R_SGD}.

\textbf{SAGA.} SAGA has its origins in the stochastic average gradient (SAG) algorithm \citep{LeRouxN2012A_FiniteSumsSG,
SchmidtM2017A_FiniteSumsSG}; moreover, the SAG algorithm is a randomized version of the incremental aggregated gradient (IAG) method proposed in \cite{BlattD2007A_IAG} and analyzed in \cite{GurbuzbalabanM2017A_IAG}. Compared with SVRG, SAGA is to apply an iteration that is closer in form to SG in that it does not compute batch gradients except possibly at the initial point, and SAGA has a practical advantage that there is only one parameter (the stepsize $\alpha$) to tune instead of two. Beyond its initialization phase, the per-iteration cost of SAGA is the same as in a SG method; but it has been shown that the method can also achieve a linear rate of convergence for strongly convex problems \citep{DefazioA2014A_SAGA}. However, the price paid to reach this rate is the need to store $n$ stochastic gradient vectors for general cases except logistic and least squares regression \citep{BottouL2018R_SGD}, which would be prohibitive in many large-scale applications.

\textbf{Iterate Averaging.} Rather than averaging the gradients, some authors propose to perform the basic SG iteration and try to use an average over iterates as the final estimator \citep{PolyakB1991A_SGIG,
PolyakB1992A_SGIG}. Since SG generates noisy iterate sequences that tend to oscillate around minimizers during the optimization process, the iterate averaging would possess less noisy behavior \citep{BottouL2018R_SGD}. It is shown that suitable iterate averaging strategies obtain an $\mathcal{O}(1/k)$ rate for strongly convex problems even for non-smooth objectives \citep{HazanE2014A_SGIG,
RakhlinA2012A_SGIG}. However, none of these methods improve on the sublinear convergence rates $\mathcal{O}(1/k)$ \citep{SchmidtM2017A_FiniteSumsSG}.

\subsection{Paper Organization}

The next section introduces the assumptions we used for establishing convergence results, especially, the variance dominant condition ( Assumption \ref{AM:ass:A3}). Then the new strategy is discussed in detail in Section \ref{AM:s3}. In Section \ref{AM:s4}, we show that under the variance dominant condition, the proposed strategy could increase the convergence rate $\mathcal{O}(1/k)$ to $\mathcal{O}(1/k^2)$ in probability for the strongly convex objectives with Lipschitz gradients, which suggests how we should control the stochastic gradient iterations to improve the rate of convergence. And we draw some conclusions in Section \ref{AM:s5}.

\section{Assumptions}
\label{AM:s2}

\subsection{Assumption on the objective}

First, let us begin with a basic assumption of smoothness of the objective function. Such an assumption is essential for convergence analyses of most gradient-based methods \citep{BottouL2018R_SGD}.

\begin{assumption}[Strongly convex objectives with Lipschitz-continuous gradients]\label{AM:ass:A1}
The objective function $F:\mathbb{R}^d\to\mathbb{R}$ is continuously differentiable and there exist $0<l\leqslant L<\infty$ such that, for all $x',x\in\mathbb{R}^d$,
\begin{align}
  \|\nabla F(x')-\nabla F(x)\|_2\leqslant L\|x'-x\|_2
  ~~\textrm{and}~~~~ \label{AM:eq:A1A}\\
  F(x')\geqslant F(x)+\nabla F(x)^\mathrm{T}(x'-x)+\frac{l}{2}\|x'-x\|_2^2. \label{AM:eq:A1B}
\end{align}
\end{assumption}

The inequality \eqref{AM:eq:A1A} ensures that the gradient of the objective $F$ is bounded and does not change arbitrarily quickly with respect to the parameter vector, which implies that
\begin{equation}\label{AM:eq:A1C}
  |F(x')-F(x)-\nabla F(x)^\mathrm{T}(x'-x)|\leqslant
  \frac{L}{2}\|x'-x\|_2^2~~\textrm{for all}~~x',x\in\mathbb{R}^d.
\end{equation}
This inequality \eqref{AM:eq:A1C} is an important basis for performing so-called mean-variance analyses for stochastic iterative sequences \citep{BottouL2018R_SGD,LuoX2019A_SGFD}. The inequality \eqref{AM:eq:A1B} is called a strong convexity condition, which is often used to ensure a sublinear convergence for the stochastic gradient methods; and the role of strong sonvexity may be essential for such rates of convergence \citep{NemirovskiA2009A_StochasticProgramming,BottouL2018R_SGD}. Under the strong sonvexity assumption, the gap between the value of the objective and the minima can be bounded by the squared $\ell_2$-norm of the gradient of the objective:
\begin{equation}\label{AM:eq:convexity}
  2l(F(x)-F_*)\leqslant\|\nabla F(x)\|_2^2
  ~~\textrm{for all}~~x\in\mathbb{R}^d.
\end{equation}
This is referred to as the Polyak-{\L}ojasiewicz inequality which was originally introduced by \cite{PolyakB1963A_Gradient}. It is a sufficient condition for gradient descent to achieve a linear convergence rate; and it is also a special case of the {\L}ojasiewicz inequality proposed in the same year \citep{LojasiewiczS1963A_PolyakGradient}, which gives an upper bound for the distance of a point to the nearest zero of a given real analytic function.

\subsection{Assumption on the variance}

We follow \cite{BottouL2018R_SGD} to make the following assumption about the variance of stochastic gradients, i.e., $g(x_k,\xi_k)$. It states that the variance of $g(x_k,\xi_k)$ is restricted in a relatively minor manner.
\begin{assumption}[Variance limit]\label{AM:ass:A2}
The objective function $F$ and the stochastic gradient $g(x_k,\xi_k)$ satisfy there exist scalars $M>0$ and $M_V>0$ such that, for all $k\in\mathbb{N}$,
\begin{equation}\label{AM:eq:A2}
  \mathbb{V}_{\xi_k}[g(x_k,\xi_k)]\leqslant M+M_V\|\nabla F(x_k)\|_2^2.
\end{equation}
\end{assumption}

\subsection{Assumption on the iteration}

Now we make the following variance dominant assumption. It states that the iterative variance $\mathbb{V}[x_j-x_i]$ is always dominant even a little bit in the stochastic gradient iterations. This assumption guarantees that the proposed strategy could achieve the convergence rate $\mathcal{O}(1/k^2)$ in probability for the strongly convex objectives.
\begin{assumption}[Variance dominant]\label{AM:ass:A3}
The sequence of iterates $\{x_k\}_{k\in\mathbb{N}}$ satisfy for all $1\leqslant j\leqslant i$, there is a fixed $\kappa>0$ (could be arbitrarily small) such that
\begin{equation}\label{AM:eq:A3}
  \|\mathbb{E}[x_j-x_i]\|_2^2=\mathcal{O}
  \big(j^{-\kappa}\mathbb{V}[x_j-x_i]\big),
\end{equation}
where $\mathbb{E}[\cdot]$ denotes the historical expectation operator defined as $\mathbb{E}[\cdot]:=\mathbb{E}_{\xi_1}\mathbb{E}_{\xi_2} \cdots\mathbb{E}_{\xi_i}[\cdot]$ and the variance $\mathbb{V}[x_j-x_i]$ is defined as 
\begin{equation*}
  \mathbb{V}[x_j-x_i]:=\mathbb{E}[\|x_j-x_i\|_2^2]
  -\|\mathbb{E}[x_j-x_i]\|_2^2.
\end{equation*}
\end{assumption}

When $\mathbb{V}[x_j-x_k]>0$, then $\|\mathbb{E}[x_j-x_k]\|_2^2$ is strictly less than $\mathbb{E}[\|x_j-x_k\|_2^2]$. And obviously, Assumption \ref{AM:ass:A3} implies that
\begin{equation}\label{AM:eq:A31}
  \|\mathbb{E}[x_j-x_i]\|_2^2=\mathcal{O}
  \big(j^{-\kappa}\mathbb{E}[\|x_j-x_i\|_2^2]\big)
\end{equation}
and 
\begin{equation*}
  \mathbb{V}[x_j-x_i]=\mathcal{O} (\mathbb{E}[\|x_j-x_i\|_2^2]).
\end{equation*}

\section{Algorithms}
\label{AM:s3}

\subsection{Methods}

Our accelerated method is procedures in which each step is chosen as a weighted average of all historical stochastic gradients. And specifically, with an initial point $x_1$, the method is characterized by the iteration
\begin{equation}\label{AM:eq:AM}
  x_{k+1}\leftarrow x_k+\alpha_km_k,
\end{equation}
where the weighted average direction
\begin{equation}\label{AM:eq:AMm}
  m_k=-\frac{1}{\sum_{i=1}^ki^p}\sum_{j=1}^kj^pg(x_j,\xi_j),~~p>0.
\end{equation}
Here, $m_k$ is the weighted average of past gradients and the values of $p$ mean different weighting methods. A larger value of $p$ makes us focus on more recent gradients, as shown in Figure \ref{AM:fig:1}; and the recommended weighting method is to choose $p=20$, which uses about the nearest $20\%$ historical gradients.

\begin{figure}[tbhp]
\centering
\subfigure{\includegraphics[width=0.45\textwidth]{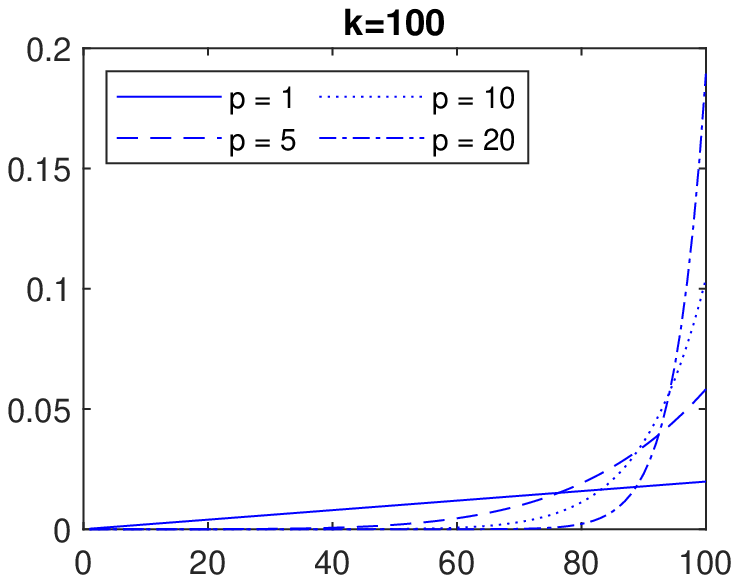}}
\subfigure{\includegraphics[width=0.45\textwidth]{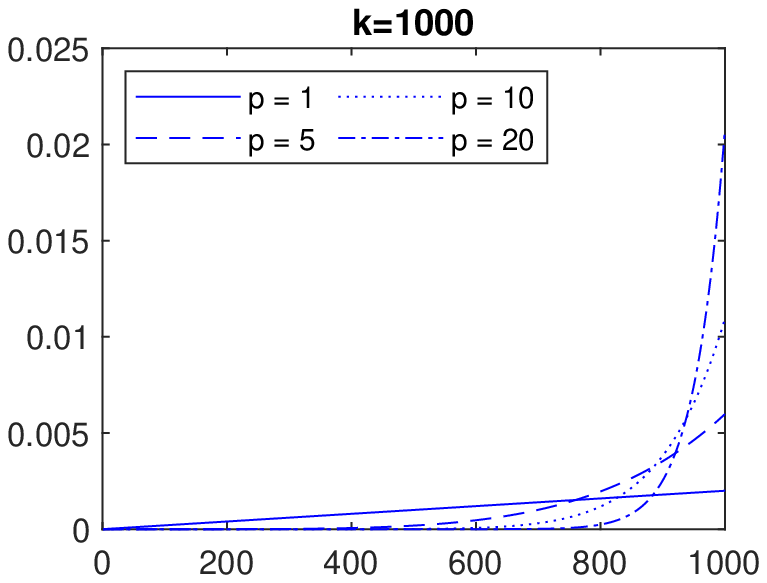}}
\caption{Illustration of the weight coefficients $\{w_j\}_{j=1}^k$ with different values of $p$ for $k=10^2$ and $10^3$, where the coefficient $w_j=\frac{j^p}{\sum_{i=1}^ki^p}$ for $j=1,\cdots,k$ and $p>0$.}
\label{AM:fig:1}
\end{figure}

Moreover, the method \eqref{AM:eq:AM} can be equivalently rewritten by the iteration
\begin{equation}\label{AM:eq:AM2}
  x_{k+1}\leftarrow x_k+\alpha_k
  \frac{v_k}{\sum_{i=1}^k\big(\frac{i}{k}\big)^p},
\end{equation}
where the direction vector $v_k$ is recursively defined as
\begin{equation*}
  v_k=\Big(\frac{k-1}{k}\Big)^pv_{k-1}-g(x_k,\xi_k)
  =-\sum_{i=1}^k\Big(\frac{i}{k}\Big)^pg(x_i,\xi_i),
\end{equation*}
which could be viewed as the classical stochastic gradient with momentum $v_k=\gamma v_{k-1}-g(x_k,\xi_k)$ where the decay factor $\gamma=\big(\frac{k-1}{k}\big)^p$ depends on $k$.

We now define our accelerated method as Algorithm \ref{AM:alg:AM}. The algorithm presumes that three computational tools exist: (i) a mechanism for generating a realization of random variable $\xi_k$ (with $\{\xi_k\}$ representing a sequence of jointly independent random variables); (ii) given an iteration number $k\in\mathbb{N}$, a mechanism for computing a scalar stepsize $\alpha_k>0$; and (iii) given an iterate $x_k\in\mathbb{R}^d$ and the realization of $\xi_k$, a mechanism for computing a stochastic vector $v_k\in\mathbb{R}^d$ and a scalar $\beta_k$.

\begin{algorithm}
\caption{Accelerated Stochastic Gradient Method}
\label{AM:alg:AM}
\begin{algorithmic}[1]
\STATE{Choose an initial iterate $x_1$.}
\FOR{$k=1,2,\cdots$}
\STATE{Generate a realization of the random variable $\xi_k$.}
\STATE{Choose a stepsize $\alpha_k>0$.}
\STATE{Compute a stochastic vector $g(x_k,\xi_k)$.}
\STATE{Update $v_k=\big(\frac{k-1}{k}\big)^pv_{k-1}-g(x_k,\xi_k)$ and $\beta_k=\sum_{i=1}^k(\frac{i}{k})^p$.}
\STATE{Set the new iterate as $x_{k+1}=x_k+\frac{\alpha_k}{\beta_k}v_k$.}
\ENDFOR
\end{algorithmic}
\end{algorithm}

\subsection{Stepsize Policy}

For strongly convex objectives, we consider the stepsize sequence $\{\alpha_k\}$ taking the form
\begin{equation}\label{AM:eq:stepsize}
  \alpha_k=\frac{s}{k+\sigma}~~\textrm{for some}~~s>\frac{4}{l}~~
  \textrm{and}~~\sigma>0~~\textrm{such that}~~
  \alpha_1\leqslant\frac{1}{LM_{G,p}^{(1)}};
\end{equation}
where the constant $M_{G,p}^{(k)}$ will be discussed in Lemma \ref{AM:lem:EmB}.

Notice that the accelerated method and the stochastic gradient method are exactly the same in the first iteration. So we assume, without loss of generality, that the first $k$ iterations $\{x_j\}_{j=1}^k$ generated by \eqref{AM:eq:AM} has the sublinear convergence rate under Assumptions \ref{AM:ass:A1} and \ref{AM:ass:A2}, that is, for every $1\leqslant j\leqslant k$, we have
\begin{equation}\label{AM:eq:CRf}
  \mathbb{E}[F(x_j)]-F_*=\mathcal{O}(1/j);
\end{equation}
then we shall prove by induction on $k$ that the accelerated method maintains the sublinear convergence rate $\mathcal{O}(1/k)$ under Assumptions \ref{AM:ass:A1} and \ref{AM:ass:A2}; and furthermore, we shall also prove that this method can achieve a convergence rate $\mathcal{O}(1/k^2)$ under Assumptions \ref{AM:ass:A1} to \ref{AM:ass:A3}.

It follows from \eqref{AM:eq:CRf} and Assumption \ref{AM:ass:A1} that
\begin{equation}\label{AM:eq:CRx}
  \mathbb{E}[\|x_j-x_*\|_2^2]\leqslant2l^{-1}
  (\mathbb{E}[F(x_j)]-F_*)=\mathcal{O}\big(j^{-1}\big).
\end{equation}
Since $1\leqslant j\leqslant k$, it follows from \eqref{AM:eq:CRx} that
\begin{align*}
  \mathbb{E}[\|x_j-x_k\|_2^2]\leqslant&
  \mathbb{E}[\|x_j-x_*\|_2+\|x_k-x_*\|_2]^2 \\
  =&\mathbb{E}[\|x_j-x_*\|_2^2+2\|x_j-x_*\|_2\|x_k-x_*\|_2
  +\|x_k-x_*\|_2^2] \\
  \leqslant&2\mathbb{E}[\|x_j-x_*\|_2^2+\|x_k-x_*\|_2^2]
  =\mathcal{O}\big(j^{-1}\big),
\end{align*}
and then, we obtain
\begin{equation}\label{AM:eq:delta}
  \|\mathbb{E}[x_j-x_k]\|_2^2\leqslant\mathbb{E}[\|x_j-x_k\|_2^2]
  =\mathcal{O}\big(j^{-1}\big).
\end{equation}
Together with Assumption \ref{AM:ass:A3}, we further obtain
\begin{equation}\label{AM:eq:delta2}
  \|\mathbb{E}[x_j-x_k]\|_2^2=\mathcal{O}
  \big(j^{-\kappa}\mathbb{E}[\|x_j-x_k\|_2^2]\big)
  =\mathcal{O}\big(j^{-1-\kappa}\big).
\end{equation}
And we will finally show that \eqref{AM:eq:delta2} implies actually $\|\mathbb{E}[x_j-x_k]\|_2^2=\mathcal{O}(j^{-2})$ in Section \ref{AM:s4}.

On the basis of \eqref{AM:eq:delta}, \eqref{AM:eq:delta2} and the stepsize policy \eqref{AM:eq:stepsize}, we first prove two Lemmas which are necessary for the following convergence analysis.
\begin{lemma}\label{AM:lem:mdecay1}
Under the conditions of \eqref{AM:eq:CRf}, suppose that the sequence of iterates $\{x_j\}_{j=1}^k$ is generated by \eqref{AM:eq:AM} with a stepsize sequence $\{\alpha_k\}$ taking the form \eqref{AM:eq:stepsize}. Then, there is $D'_p<\infty$ such that for any given diagonal matrix $\Lambda= \mathrm{diag}(\lambda_1,\cdots,\lambda_d)$ with $\lambda_i\in [-L,L]$, the inequality
\begin{equation}\label{AM:eq:mdecay1}
  \frac{1}{\sum_{i=1}^ki^p}
  \bigg\|\sum_{j=1}^kj^p\Lambda(x_j-x_k)\bigg\|_2
  \leqslant\frac{\sqrt{\alpha_k}LD'_p}{2}
\end{equation}
holds in probability.
\end{lemma}
\begin{proof}
Note that for any $a\in\mathbb{R}$,
\begin{equation}\label{AM:eq:localOrder}
  \mathcal{O}\bigg(\sum_{i=1}^ki^a\bigg)=\int_0^kt^a\ud t=\frac{k^{a+1}}{a+1},
\end{equation}
together with \eqref{AM:eq:stepsize}, we have
$\sqrt{\alpha_k}\sum_{i=1}^ki^p=\mathcal{O}\big(k^{p+\frac{1}{2}}\big)$; thus, to prove \eqref{AM:eq:mdecay1}, we only need to show that
\begin{equation*}
  \bigg\|\sum_{j=1}^kj^p\Lambda(x_j-x_k)\bigg\|_2
  =\mathcal{O}\big(k^{p+\frac{1}{2}}\big)
\end{equation*}
holds in probability. Using the mean-variance analysis, it follows from $\|\mathbb{E}[x_j-x_k]\|_\infty\leqslant\|\mathbb{E}[x_j-x_k]\|_2$ and \eqref{AM:eq:delta} that
\begin{equation*}
  \mathbb{E}\bigg[\sum_{j=1}^kj^p\Lambda(x_j-x_k)\bigg]
  =\Lambda\sum_{j=1}^kj^{p}~\mathbb{E}[x_j-x_k]
  =\mathcal{O}\bigg(\Lambda\sum_{j=1}^kj^{p-\frac{1}{2}}\bigg)
  =\mathcal{O}\big(k^{p+\frac{1}{2}}\big);
\end{equation*}
meanwhile, according to \eqref{AM:eq:delta}, we also have
\begin{equation}\label{AM:eq:localV}
  \mathbb{V}\bigg[\sum_{j=1}^kj^p\Lambda(x_j-x_k)\bigg]
  =\Lambda^2\sum_{j=1}^kj^{2p}~\mathbb{V}[x_j-x_k]
  \leqslant\Lambda^2\sum_{j=1}^kj^{2p-1}=\mathcal{O}(k^{2p}).
\end{equation}
Using Chebyshev's inequality, there is a $C>0$ such that for $\epsilon>0$,
\begin{equation*}
  \mathbb{P}\bigg(\bigg\|\sum_{j=1}^kj^p\Lambda(x_j-x_k)
  -Ck^{p+\frac{1}{2}}\Lambda\bigg\|_2\geqslant\epsilon
  k^p\Lambda\bigg)\leqslant\frac{1}{\epsilon^2},
\end{equation*}
which gives the inequality \eqref{AM:eq:mdecay1} in probability.
\end{proof}

Under Assumption \ref{AM:ass:A3}, it is clear that \eqref{AM:eq:mdecay1} could be further strengthened.
\begin{lemma}\label{AM:lem:mdecay2}
Suppose the conditions of Lemma \ref{AM:lem:mdecay1} and Assumption \ref{AM:ass:A3} hold. Then, there is $D_p<\infty$ such that for any given diagonal matrix $\Lambda= \mathrm{diag}(\lambda_1,\cdots,\lambda_d)$ with $\lambda_i\in [-L,L]$, the inequality
\begin{equation}\label{AM:eq:mdecay2}
  \frac{1}{\sum_{i=1}^ki^p}
  \bigg\|\sum_{j=1}^kj^p\Lambda(x_j-x_k)\bigg\|_2
  \leqslant\frac{\alpha_k^sLD_p}{2}
\end{equation}
holds in probability, where $s=\min\big(1,\frac{1+\kappa}{2}\big)$.
\end{lemma}
\begin{proof}
Note that \eqref{AM:eq:localOrder}, we have $\alpha_k^s\sum_{i=1}^ki^p= \mathcal{O}\big(k^{p+1-s}\big)$; thus, to prove \eqref{AM:eq:mdecay2}, we only need to show that
\begin{equation*}
  \bigg\|\sum_{j=1}^kj^p\Lambda(x_j-x_k)\bigg\|_2
  =\mathcal{O}\big(k^{p+1-s}\big)
\end{equation*}
holds in probability. First, it follows from $\|\mathbb{E}[x_j-x_k]\|_\infty\leqslant\|\mathbb{E}[x_j-x_k]\|_2$ and \eqref{AM:eq:delta2} that
\begin{equation*}
  \mathbb{E}\bigg[\sum_{j=1}^kj^p\Lambda(x_j-x_k)\bigg]
  =\Lambda\sum_{j=1}^kj^{p}~\mathbb{E}[x_j-x_k]
  =\mathcal{O}\bigg(\Lambda\sum_{j=1}^kj^{p-\frac{1+\kappa}{2}}\bigg)
  =\mathcal{O}\big(k^{p+\frac{1-\kappa}{2}}\big);
\end{equation*}
together with \eqref{AM:eq:localV} and using Chebyshev's inequality, there is a $C>0$ such that
\begin{equation*}
  \mathbb{P}\bigg(\bigg\|\sum_{j=1}^kj^p\Lambda(x_j-x_k)
  -Ck^{p+\frac{1-\kappa}{2}}\Lambda\bigg\|_2
  \geqslant\epsilon Vk^p\Lambda\bigg)\leqslant\frac{1}{\epsilon^2}.
\end{equation*}
It is worth noting that when $\kappa\geqslant1$, the variance will become the principal part; so the proof is complete.
\end{proof}

\section{Convergence Results}
\label{AM:s4}

\subsection{Mean-Variance Framework}

The mean-variance framework can be described as a fundamental lemma for any iteration based on random steps, which relies only on Assumption \ref{AM:ass:A1} and is a slight generalization of Lemma 4.2 in \cite{BottouL2018R_SGD}.

\begin{lemma}\label{AM:lem:MVF}
Under Assumption \ref{AM:ass:A1}, if for every $k\in\mathbb{N}$, $\xi_k$ is any random vector independent of $x_k$ and $s(x_k,\xi_k)$ is a stochastic step depending on $\xi_k$, then the iteration
\begin{equation*}
  x_{k+1}=x_k+s(x_k,\xi_k)
\end{equation*}
satisfy the following inequality
\begin{align*}
  \mathbb{E}_{\xi_k}[F(x_{k+1})]-F(x_k)
  \leqslant&\nabla F(x_k)^\mathrm{T}\mathbb{E}_{\xi_k}[s(x_k,\xi_k)]
  +\frac{L}{2}\|\mathbb{E}_{\xi_k}[s(x_k,\xi_k)]\|_2^2
  +\frac{L}{2}\mathbb{V}_{\xi_k}[s(x_k,\xi_k)],
\end{align*}
where the variance of $s(x_k,\xi_k)$ is defined as
\begin{equation}\label{AM:eq:V}
  \mathbb{V}_{\xi_k}[s(x_k,\xi_k)]:=
  \mathbb{E}_{\xi_k}[\|s(x_k,\xi_k)\|_2^2]
  -\|\mathbb{E}_{\xi_k}[s(x_k,\xi_k)]\|_2^2.
\end{equation}
\end{lemma}
\begin{proof}
According to the inequality \eqref{AM:eq:A1C}, the iteration $x_{k+1}=x_k+s(x_k,\xi_k)$ satisfy
\begin{align*}
  F(x_{k+1})-F(x_k)\leqslant&\nabla F(x_k)^\mathrm{T}(x_{k+1}-x_k)
  +\frac{L}{2}\|x_{k+1}-x_k\|_2^2 \\
  \leqslant&\nabla F(x_k)^\mathrm{T}s(x_k,\xi_k)
  +\frac{L}{2}\|s(x_k,\xi_k)\|_2^2.
\end{align*}
Noting that $\xi_k$ is independent of $x_k$ and taking expectations in these inequalities with respect to the distribution of $\xi_k$, we obtain
\begin{equation*}
  \mathbb{E}_{\xi_k}[F(x_{k+1})]-F(x_k)\leqslant\nabla F(x_k)^\mathrm{T}\mathbb{E}_{\xi_k}[s(x_k,\xi_k)]
  +\frac{L}{2}\mathbb{E}_{\xi_k}[\|s(x_k,\xi_k)\|_2^2].
\end{equation*}
Recalling \eqref{AM:eq:V}, we finally get the desired bound.
\end{proof}

Regardless of the states before $x_k$, the expected decrease in the objective function yielded by the $k$th stochastic step $s(x_k,\xi_k)$, say, $\mathbb{E}_{\xi_k}[F(x_{k+1})]-F(x_k)$, could be bounded above by a quantity involving the expectation $\mathbb{E}_{\xi_k}[s(x_k,\xi_k)]$ and variance $\mathbb{V}_{\xi_k}[s(x_k,\xi_k)]$.

\subsection{Expectation Analysis}

Now we will analyze the mean of $m_k$ to get the bounds of $\nabla F(x_k)^\mathrm{T}\mathbb{E}_{\xi_k}[m_k]$ and $\|\mathbb{E}_{\xi_k}[m_k]\|_2^2$, where $\mathbb{E}[m_k]=\mathbb{E}_{\xi_1} \cdots\mathbb{E}_{\xi_k}[m_k]$ is the historical expectation of $m_k$. First, according to the definition \eqref{AM:eq:AMm} of the weighted average direction $m_k$, we have
\begin{equation*}
  \mathbb{E}[m_k]=-\frac{1}{\sum_{i=1}^ki^p}
  \sum_{j=1}^kj^p\mathbb{E}_{\xi_j}[g(x_j,\xi_j)]
  =-\frac{1}{\sum_{i=1}^ki^p}\sum_{j=1}^kj^p\nabla F(x_j).
\end{equation*}
Further, from Assumption \ref{AM:ass:A1}, we have
\begin{equation*}
  \|\nabla F(x_j)-\nabla F(x_k)\|_\infty\leqslant
  \|\nabla F(x_j)-\nabla F(x_k)\|_2\leqslant L\|x_j-x_k\|_2,
\end{equation*}
then there is a diagonal matrix $\Lambda=\mathrm{diag}(\lambda_1,\cdots, \lambda_d)$ with $\lambda_i\in [-L,L]$ such that
\begin{equation}\label{AM:eq:AMDj}
  \nabla F(x_j)=\nabla F(x_k)+\Lambda(x_j-x_k).
\end{equation}
Therefore, $\mathbb{E}[m_k]$ could be written as
\begin{equation}\label{AM:eq:AMm2}
  \mathbb{E}[m_k]=-\nabla F(x_k)+
  \frac{1}{\sum_{i=1}^ki^p}\sum_{j=1}^kj^p\Lambda(x_k-x_j),
\end{equation}
where $\Lambda=\mathrm{diag}(\lambda_1,\cdots,\lambda_d)$ is a diagonal matrix  with $\lambda_i\in [-L,L]$.

Together with Lemma \ref{AM:lem:mdecay1}, we get the following bounds of $\nabla F(x_k)^\mathrm{T}\mathbb{E}[m_k]$ and $\|\mathbb{E}[m_k]\|_2^2$.
\begin{theorem}\label{AM:thm:momentum}
Suppose the conditions of Lemma \ref{AM:lem:mdecay1} hold. Then for every $k\in\mathbb{N}$, the following conditions
\begin{align}
  \|\mathbb{E}[m_k]\|_2\leqslant&\|\nabla F(x_k)\|_2+
  \frac{\sqrt{\alpha_k}LD'_p}{2}
  ~~\textrm{and}~~\label{AM:eq:mstep1} \\
  \nabla F(x_k)^\mathrm{T}\mathbb{E}[m_k]
  \leqslant&-\|\nabla F(x_k)\|_2^2+\frac{\sqrt{\alpha_k}LD'_p}{2}
  \|\nabla F(x_k)\|_2 \label{AM:eq:mstep2}
\end{align}
hold in probability.
\end{theorem}
\begin{remark}\label{AM:rem:GA}
For the gradient averaging method mentioned \eqref{AM:eq:GA} in Subsection \ref{AM:s11},  it is using the average of all previous gradients,
\begin{equation*}
  x_{k+1}=x_k+\alpha_km'_k,
\end{equation*}
where 
\begin{equation*}
  m'_k=-\frac{1}{k}\sum_{i=1}^kg(x_i,\xi_i).
\end{equation*}
By \eqref{AM:eq:AMDj}, the historical expectation of $m'_k$ could be written as
\begin{equation*}
  \mathbb{E}[m'_k]=-\frac{1}{k}\sum_{j=1}^k
  \mathbb{E}_{\xi_j}[g(x_j,\xi_j)]
  =-\frac{1}{k}\sum_{j=1}^k\nabla F(x_j)=-\nabla F(x_k)+R',
\end{equation*}
where 
\begin{equation*}
  R'=\frac{1}{k}\sum_{j=1}^k\Lambda(x_k-x_j).
\end{equation*}
And the bound of $R'$ is $\frac{L}{k}\sum_{j=1}^kj^{-1/2}$, which decays slower than $k^{-1/2}$, i.e., $\mathcal{O}(\sqrt{\alpha_k})$ described in Theorem \ref{AM:thm:momentum}.
\end{remark}
\begin{proof}
According to \eqref{AM:eq:AMm2} and Lemma \ref{AM:lem:mdecay1}, it follows that
\begin{equation*}
  \mathbb{E}[m_k]=-\nabla F(x_k)+R,
\end{equation*}
where the vector $R=\frac{1}{\sum_{i=1}^ki^p}\sum_{j=1}^kj^p\Lambda(x_k-x_j)$ and $\|R\|_2\leqslant\frac{\sqrt{\alpha_k}LD'_p}{2}$. Thus, one obtains \eqref{AM:eq:mstep1} and \eqref{AM:eq:mstep2} by noting that
\begin{equation*}
  \|\mathbb{E}[m_k]\|_2\leqslant\|\nabla F(x_k)\|_2+\|R\|_2
\end{equation*}
and
\begin{equation*}
  \nabla F(x_k)^\mathrm{T}\mathbb{E}[m_k]=
  -\|\nabla F(x_k)\|_2^2+\nabla F(x_k)^\mathrm{T}R
  \leqslant-\|\nabla F(x_k)\|_2^2+\|\nabla F(x_k)\|_2\|R\|_2,
\end{equation*}
so the proof is complete.
\end{proof}

Under Assumption \ref{AM:ass:A3}, both \eqref{AM:eq:mstep1} and \eqref{AM:eq:mstep2} can be further improved.
\begin{theorem}\label{AM:thm:momentum2}
Suppose the conditions of Lemma \ref{AM:lem:mdecay2} hold. Then for every $k\in\mathbb{N}$, the following conditions
\begin{align}
  \|\mathbb{E}[m_k]\|_2\leqslant&\|\nabla F(x_k)\|_2+\frac{\alpha_k^sLD_p}{2}
  ~~\textrm{and}~~\label{AM:eq:mstep3} \\
  \nabla F(x_k)^\mathrm{T}\mathbb{E}[m_k]
  \leqslant&-\|\nabla F(x_k)\|_2^2+\frac{\alpha_k^sLD_p}{2}
  \|\nabla F(x_k)\|_2 \label{AM:eq:mstep4}
\end{align}
hold in probability, where $s=\min\big(1,\frac{1+\kappa}{2}\big)$.
\end{theorem}
\begin{proof}
According to Lemma \ref{AM:lem:mdecay2}, we have $\frac{1}{\sum_{i=1}^ki^p}\|\sum_{j=1}^kj^p\Lambda(x_k-x_j)\|_2 \leqslant\frac{\alpha_k^sLD_p}{2}$, and
the desired results could be proved in the same way as the proof of Theorem \ref{AM:thm:momentum}.
\end{proof}

\subsection{Variance Analysis}

Now we will analyze the variance of $m_k$ to get the bound of $\mathbb{V}[m_k]$. As an important result, we will show that the variance of $m_k$ tends to zero with a rate $\mathcal{O}(k^{-1})$ as $k$ grows.
\begin{lemma}\label{AM:lem:momentumV}
Under Assumption \ref{AM:ass:A2}, suppose that the sequence of iterates $\{x_k\}$ is generated by \eqref{AM:eq:AM} with $\|x_i-x_j\|_2\leqslant D$ for any $i,j\in\mathbb{N}$. Then \eqref{AM:eq:stepsize}, then
\begin{equation}\label{AM:eq:momentumVC}
  \mathbb{V}[m_k]\leqslant C_p\alpha_k
  \Big(M+2M_V\|\nabla F(x_k)\|_2^2+2M_VL^2D^2\Big),
\end{equation}
where $C_p$ is positive real constant.
\end{lemma}
\begin{remark}\label{AM:rem:SGM}
For the SGM method \eqref{AM:eq:SGM} mentioned in Subsection \ref{AM:s11},  it is using the weighted average of all previous gradients,
\begin{equation*}
  x_{k+1}=x_k+\alpha_km''_k,
\end{equation*}
where
\begin{equation*}
  m''_k=-\frac{1}{\sum_{i=1}^k\beta^{k-i}}
  \sum_{i=1}^k\beta^{k-i}g(x_i,\xi_i).
\end{equation*}
Since 
\begin{equation*}
  \frac{1}{\big(\sum_{i=1}^k\beta^{k-i}\big)^2}
  \sum_{j=1}^k\beta^{2(k-j)}
  =\frac{1-\beta^{2k}}{1-\beta^2}\frac{(1-\beta)^2}{(1-\beta^k)^2}
  =\frac{1-\beta}{1+\beta}\frac{1+\beta^k}{1-\beta^k},
\end{equation*}
it follows that 
\begin{equation*}
  \mathbb{V}[m''_k]=\frac{1}{\big(\sum_{i=1}^k\beta^{k-i}\big)^2}
  \sum_{j=1}^k\beta^{2(k-j)}~\mathbb{V}_{\xi_j}[g(x_j,\xi_j)].
\end{equation*}
Together with the proof below, one can find that the variance of $m''_k$ could be controlled by a by a fixed fraction $\frac{1-\beta}{1+\beta}$. 
\end{remark}
\begin{proof}
It follows from \eqref{AM:eq:AMDj} and $\|x_j-x_k\|_2\leqslant D$ that
\begin{equation*}
  \|\nabla F(x_j)\|_2\leqslant\|\nabla F(x_k)\|_2+L\|x_j-x_k\|_2
  \leqslant\|\nabla F(x_k)\|_2+LD,
\end{equation*}
together with the Arithmetic Mean Geometric Mean inequality, we have
\begin{align*}
  \|\nabla F(x_j)\|_2^2\leqslant&\Big(\|\nabla F(x_k)\|_2+LD\Big)^2 \\
  \leqslant&\|\nabla F(x_k)\|_2^2+L^2D^2+2LD\|\nabla F(x_k)\|_2 \\
  \leqslant&2\|\nabla F(x_k)\|_2^2+2L^2D^2.
\end{align*}
Hence, along with Assumption \ref{AM:ass:A2}, we obtain
\begin{align*}
  \mathbb{V}[m_k]=&\frac{1}{\big(\sum_{i=1}^ki^p\big)^2}
  \sum_{j=1}^kj^{2p}~\mathbb{V}_{\xi_j}[g(x_j,\xi_j)] \\
  \leqslant&\frac{1}{\big(\sum_{i=1}^ki^p\big)^2}
  \sum_{j=1}^kj^{2p}\Big(M+M_V\|\nabla F(x_j)\|_2^2\Big) \\
  \leqslant&\frac{\sum_{i=1}^kj^{2p}}{\big(\sum_{i=1}^ki^p\big)^2}
  \Big(M+2M_V\|\nabla F(x_k)\|_2^2+2M_VL^2D^2\Big).
\end{align*}
According to \eqref{AM:eq:localOrder}, $\sum_{j=1}^kj^{2p}=\mathcal{O} (k^{2p+1})$ and $\big(\sum_{i=1}^ki^p\big)^2=\mathcal{O}(k^{2p+2})$; therefore, we have $\mathbb{V}[m_k]=\mathcal{O}(k^{-1}) =\mathcal{O}(\alpha_k)$, and the proof is complete.
\end{proof}

Combining Theorem \ref{AM:thm:momentum} and Lemma \ref{AM:lem:momentumV}, we can obtain a bound for each iteration of the accelerated method.
\begin{lemma}\label{AM:lem:EmB}
Under the conditions of Theorem \ref{AM:thm:momentum} and Lemma \ref{AM:lem:momentumV}, suppose that the stepsize sequence $\{\alpha_k\}$ satisfies $\alpha_k\leqslant\frac{1}{L}$. Then, the inequality
\begin{equation*}
  \mathbb{E}_{\xi_k}[F(x_{k+1})]-F(x_k)\leqslant
  -\frac{3\alpha_k}{4}\|\nabla F(x_k)\|_2^2
  +\frac{\alpha_k^2LM_{G,p}^{(k)}}{2}\|\nabla F(x_k)\|_2^2
  +\frac{\alpha_k^2LM_{d,p,1}^{(k)}}{2}
\end{equation*}
holds in probability, where $M_{G,p}^{(k)}=\frac{3}{2}+2C_p\alpha_kM_V$ and $M_{d,p,1}^{(k)}=\frac{5L{D'_p}^2}{4}+C_p\alpha_k(M+2M_VL^2D^2)$; further,
\begin{equation*}
  \lim_{k\to\infty}M_{G,p}^{(k)}=\frac{3}{2}~~\textrm{and}~~
  \lim_{k\to\infty}M_{d,p,1}^{(k)}=\frac{5L{D'_p}^2}{4}.
\end{equation*}
\end{lemma}
\begin{proof}
According to \eqref{AM:eq:mstep1}, together with the Arithmetic Mean Geometric Mean inequality, one obtains
\begin{align*}
  \left\|\mathbb{E}_{\xi_k}[m_k]\right\|_2^2
  \leqslant&\bigg(\|\nabla F(x_k)\|_2+
  \frac{\sqrt{\alpha_k}LD'_p}{2}\bigg)^2 \\
  \leqslant&\|\nabla F(x_k)\|_2^2
  +\sqrt{\alpha_k}LD'_p\|\nabla F(x_k)\|_2
  +\frac{\alpha_kL^2{D'_p}^2}{4} \\
  \leqslant&\frac{3}{2}\|\nabla F(x_k)\|_2^2
  +\frac{3\alpha_kL^2{D'_p}^2}{4}.
\end{align*}
Similarly, by \eqref{AM:eq:mstep2} and the Arithmetic Mean Geometric Mean inequality, it holds that
\begin{align*}
  \nabla F(x_k)^\mathrm{T}\mathbb{E}_{\xi_k}[m_k]
  \leqslant&-\|\nabla F(x_k)\|_2^2+
  \frac{\sqrt{\alpha_k}LD'_p}{2}\|\nabla F(x_k)\|_2 \\
  \leqslant&-\frac{3}{4}\|\nabla F(x_k)\|_2^2
  +\frac{\alpha_kL^2{D'_p}^2}{4}.
\end{align*}
Finally, accoding to Lemma \ref{AM:lem:MVF}, Assumption \ref{AM:ass:A2}, and $\alpha_k\leqslant\frac{1}{L}$, the iterates satisfy
\begin{align*}
  \mathbb{E}_{\xi_k}[F(x_{k+1})]-F(x_k)\leqslant&
  \nabla F(x_k)^\mathrm{T}\mathbb{E}_{\xi_k}[\alpha_km_k]
  +\frac{L}{2}\|\mathbb{E}_{\xi_k}[\alpha_km_k]\|_2^2
  +\frac{L}{2}\mathbb{V}_{\xi_k}[\alpha_km_k] \\
  =&\alpha_k\nabla F(x_k)^\mathrm{T}\mathbb{E}_{\xi_k}[m_k]
  +\frac{\alpha_k^2L}{2}\|\mathbb{E}_{\xi_k}[m_k]\|_2^2
  +\frac{\alpha_k^2L}{2}\mathbb{V}_{\xi_k}[m_k] \\
  \leqslant&-\frac{3\alpha_k}{4}\|\nabla F(x_k)\|_2^2
  +\alpha_k^2L\left(\frac{3}{4}+\frac{C_p}{k}M_V\right)\|\nabla F(x_k)\|_2^2 \\
  &+\frac{\alpha_k^2L}{2}\bigg(\frac{5L{D'_p}^2}{4}
  +\frac{C_p}{k}(M+2M_VL^2D^2)\bigg) \\
  =&-\alpha_k\|\nabla F(x_k)\|_2^2
  +\frac{\alpha_k^2LM_{G,p}^{(k)}}{2}\|\nabla F(x_k)\|_2^2
  +\frac{\alpha_k^2LM_{d,p,1}^{(k)}}{2},
\end{align*}
and the proof is complete.
\end{proof}

From Theorem \ref{AM:thm:momentum2}, the bound we obtained could be further improved in the same way as the proof of Lemma \ref{AM:lem:EmB}.
\begin{lemma}\label{AM:lem:EmB2}
Under the conditions of Theorem \ref{AM:thm:momentum2} and Lemma \ref{AM:lem:momentumV}, suppose that the stepsize sequence $\{\alpha_k\}$ satisfies $\alpha_k\leqslant\frac{1}{L}$. Then, the inequality
\begin{equation*}
  \mathbb{E}_{\xi_k}[F(x_{k+1})]-F(x_k)\leqslant
  -\frac{3\alpha_k}{4}\|\nabla F(x_k)\|_2^2
  +\frac{\alpha_k^2LM_{G,p}^{(k)}}{2}\|\nabla F(x_k)\|_2^2
  +\frac{\alpha_k^{2+\kappa'}LM_{d,p,2}^{(k)}}{2}
\end{equation*}
holds in probability, where $\kappa'=\min(1,\kappa)$, $M_{G,p}^{(k)}$ comes from \ref{AM:lem:EmB}, and
\begin{equation*}
  M_{d,p,2}^{(k)}=\frac{5}{4}LD_p^2
  +C_p\alpha_k^{1-\kappa'}(M+2M_VL^2D^2)\leqslant
  \frac{5}{4}LD_p^2+\frac{C_p}{L^{1-\kappa'}}(M+2M_VL^2D^2).
\end{equation*}
\end{lemma}

\subsection{Average Behavior of Iterations}

According to Lemmas \ref{AM:lem:EmB} and \ref {AM:lem:EmB2}, it is easy to analyze the average behavior of iterations of the accelerated method for strong convex functions.
\begin{theorem}\label{AM:thm:ABm}
Under the conditions of Lemma \ref{AM:lem:EmB}, suppose that the stepsize sequence $\{\alpha_k\}$ satisfies $\alpha_k\leqslant\frac{1}{LM_{G,p}^{(1)}}$. Then, the inequality
\begin{equation*}
  \mathbb{E}[F(x_{k+1})-F_*]\leqslant
  \left[\prod_{i=1}^k\Big(1-\frac{\alpha_il}{2}\Big)\right]
  [F(x_1)-F_*]+\frac{LM_{d,p,1}^{(1)}}{2}\sum_{i=1}^k
  \alpha_i^2\prod_{j=i+1}^k\Big(1-\frac{\alpha_il}{2}\Big)
\end{equation*}
holds in probability.
\end{theorem}
\begin{proof}
According to Lemma \ref{AM:lem:EmB} and $0<\alpha_k\leqslant \frac{1}{LM_{G,p}^{(1)}}\leqslant\frac{1}{LM_{G,p}^{(k)}}$, we have
\begin{align*}
  \mathbb{E}_{\xi_k}[F(x_{k+1})]\leqslant&
  F(x_k)-\frac{3\alpha_k}{4}\|\nabla F(x_k)\|_2^2
  +\frac{\alpha_k^2LM_{G,p}^{(k)}}{2}\|\nabla F(x_k)\|_2^2
  +\frac{\alpha_k^2LM_{d,p,1}^{(k)}}{2} \\
  \leqslant&F(x_k)-\frac{\alpha_k}{4}\|\nabla F(x_k)\|_2^2
  +\frac{\alpha_k^2LM_{d,p,1}^{(k)}}{2}.
\end{align*}
Subtracting $F_*$ from both sides and applying \eqref{AM:eq:convexity}, this yields
\begin{align*}
  \mathbb{E}_{\xi_k}[F(x_{k+1})]-F_*
  \leqslant&F(x_k)-F_*-\frac{\alpha_k}{4}\|\nabla F(x_k)\|_2^2
  +\frac{\alpha_k^2LM_{d,p,1}^{(k)}}{2} \\
  \leqslant&F(x_k)-F_*-\frac{\alpha_kl}{2}\big(F(x_k)-F_*\big)
  +\frac{\alpha_k^2LM_{d,p,1}^{(k)}}{2} \\
  =&\Big(1-\frac{\alpha_kl}{2}\Big)\big(F(x_k)-F_*\big)
  +\frac{\alpha_k^2LM_{d,p,1}^{(k)}}{2},
\end{align*}
and it follows from taking historical expectations that
\begin{align*}
  \mathbb{E}[F(x_{k+1})-F_*]\leqslant&\Big(1-\frac{\alpha_kl}{2}\Big)
  \mathbb{E}[F(x_k)-F_*]+\frac{\alpha_k^2LM_{d,p,1}^{(k)}}{2} \\
  \leqslant&\Big(1-\frac{\alpha_kl}{2}\Big)
  \mathbb{E}[F(x_k)-F_*]+\frac{\alpha_k^2LM_{d,p,1}^{(1)}}{2}.
\end{align*}
Therefore, the desired result follows by repeatedly applying this inequality above through iteration from $1$ to $k$.
\end{proof}

According to Lemma \ref{AM:lem:EmB2}, the result could be further improved.
\begin{theorem}\label{AM:thm:ABm2}
Under the conditions of Lemma \ref{AM:lem:EmB2}, suppose that the stepsize sequence $\{\alpha_k\}$ satisfies $\alpha_k\leqslant\frac{1}{LM_{G,p}^{(1)}}$. Then, the inequality
\begin{equation*}
  \mathbb{E}[F(x_{k+1})-F_*]\leqslant
  \left[\prod_{i=1}^k\Big(1-\frac{\alpha_il}{2}\Big)\right]
  [F(x_1)-F_*]+\frac{LM_{d,p,2}^{(1)}}{2}\sum_{i=1}^k
  \alpha_i^{2+\kappa'}\prod_{j=i+1}^k\Big(1-\frac{\alpha_il}{2}\Big)
\end{equation*}
holds in probability, where $\kappa'=\min(1,\kappa)$.
\end{theorem}

\subsection{Convergence}

According to Theorems \ref{AM:thm:ABm} and \ref{AM:thm:ABm2}, the convergence of the accelerated methods is closely related to the following two limits:
\begin{equation}\label{AM:eq:limits}
  A:=\lim_{k\to\infty}A_k~~~\textrm{and}~~~
  B(a):=\lim_{k\to\infty}B_k(a),
\end{equation}
where $A_k=\prod_{i=1}^k\big(1-\frac{\alpha_il}{2}\big)$, $B_k(a)=\sum_{i=1}^k\alpha_i^{2+a}\prod_{j=i+1}^k \big(1-\frac{\alpha_il}{2}\big)$ and $0\leqslant a\leqslant1$. Therefore, the results in Theorems \ref{AM:thm:ABm} and \ref{AM:thm:ABm2} can be rewritten as
\begin{equation}\label{AM:eq:R1r}
  \mathbb{E}[F(x_{k+1})-F_*]\leqslant[F(x_1)-F_*]A_k
  +\frac{LM_{d,p,1}^{(1)}}{2}B_k(0)
\end{equation}
and
\begin{equation}\label{AM:eq:R2r}
  \mathbb{E}[F(x_{k+1})-F_*]\leqslant[F(x_1)-F_*]A_k
  +\frac{LM_{d,p,2}^{(1)}}{2}B_k(\kappa'),
\end{equation}
respectively. In the following, we will use the properties of the gamma function to analyze the asymptotic behavior of $A_k$ and $B_k(a)$.

Let $\Gamma(z)$ denote the gamma function for all $z\neq0,-1,-2,\cdots$, and let $(z)_t$ denote the Pochhammer symbol or shifted factorial $z(z+1)\cdots(z+t-1)$ for all $t=1,2,\cdots$, then we have two recursive formulas $z\Gamma(z)=\Gamma(1+z)$ and
\begin{equation}\label{AM:eq:gamma}
  (z)_t\Gamma(z)=\Gamma(t+z).
\end{equation}
And we also have the following lemma which gives the first-order asymptotic expansion of the ratio of two gamma functions:
\begin{lemma}[\cite{TricomiF1951A_GammaRatio}]
\label{AM:lem:GammaRatio}
For any $a\in\mathbb{R}$,
\begin{equation*}
  \frac{\Gamma(x+a)}{\Gamma(x)}=x^a+\mathcal{O}\left(x^{a-1}\right)
\end{equation*}
as $x\to\infty$.
\end{lemma}

Now we can prove the following first-order asymptotic expansions of $A_k$ and $B_k$:
\begin{lemma}\label{AM:lem:ABk}
If a stepsize sequence takes the form \eqref{AM:eq:stepsize}, then we have the following first-order asymptotic expansions
\begin{equation*}
  A_k=\frac{\Gamma(1+\sigma)}{\Gamma(1+\sigma-\frac{sl}{2})}
  (k+1+\sigma)^{-\frac{sl}{2}}
  +\mathcal{O}\left((k+1+\sigma)^{-1-\frac{sl}{2}}\right)
\end{equation*}
and
\begin{equation*}
  B_k(a)=\frac{Cs^{2+a}}{\frac{sl}{2}\!-\!1\!-\!a}(k+1+\sigma)^{-1-a}
  +\mathcal{O}\left((k+1+\sigma)^{-2-a}\right)
\end{equation*}
as $k\to\infty$, where $C$ is a positive real number.
\end{lemma}
\begin{proof}
According to the stepsize policy \eqref{AM:eq:stepsize}, i.e.,
\begin{equation*}
  \alpha_k=\frac{s}{k+\sigma}~~\textrm{for some}~~s>\frac{4}{l}~~
  \textrm{and}~~\sigma>0~~\textrm{such that}~~
  \alpha_1\leqslant\frac{1}{LM_{G,p}^{(1)}},
\end{equation*}
we have
\begin{equation*}
  1-\frac{\alpha_il}{2}=\frac{i+\sigma-\frac{sl}{2}}{i+\sigma}
  ~~\textrm{with}~~\frac{sl}{2}>2.
\end{equation*}
By using the Pochhammer symbol, $A_k$ can be written as
\begin{equation*}
  A_k=\prod_{i=1}^k\Big(1-\frac{\alpha_il}{2}\Big)
  =\prod_{i=1}^k\frac{i+\sigma-\frac{sl}{2}}{i+\sigma}
  =\frac{(1+\sigma-\frac{sl}{2})_k}{(1+\sigma)_k},
\end{equation*}
together with the recursive formula \eqref{AM:eq:gamma}, $A_k$ can be further written as
\begin{equation*}
  A_k=\frac{(1+\sigma-\frac{sl}{2})_k}{(1+\sigma)_k}
  =\frac{\Gamma(1+\sigma)}{\Gamma(1+\sigma-\frac{sl}{2})}
  \frac{\Gamma(k+1+\sigma-\frac{sl}{2})}{\Gamma(k+1+\sigma)},
\end{equation*}
then the first-order asymptotic expansion of $A_k$ can be obtained from Lemma \ref{AM:lem:GammaRatio}.

Similarly, according to the stepsize policy \eqref{AM:eq:stepsize}, we have
\begin{equation*}
  B_k(a)=\sum_{i=1}^k\frac{s^{2+a}}{(i+\sigma)^{2+a}}
  \prod_{j=i+1}^k\frac{j+\sigma-\frac{sl}{2}}{j+\sigma}.
\end{equation*}
By using the Pochhammer symbol, the cumulative product term in the sum above can be written as
\begin{equation*}
  \prod_{j=i+1}^k\frac{j+\sigma-sl}{j+\sigma}
  =\frac{(i+1+\sigma-\frac{sl}{2})_{k-i}}{(i+1+\sigma)_{k-i}},
\end{equation*}
together with the recursive formula \eqref{AM:eq:gamma}, this cumulative product term above can be further written as
\begin{equation*}
  \prod_{j=i+1}^k\frac{j+\sigma-sl}{j+\sigma}
  =\frac{(i+1+\sigma-\frac{sl}{2})_{k-i}}{(i+1+\sigma)_{k-i}}
  =\frac{\Gamma(k+1+\sigma-\frac{sl}{2})}{\Gamma(k+1+\sigma)}
  \frac{\Gamma(i+1+\sigma)}{\Gamma(i+1+\sigma-\frac{sl}{2})},
\end{equation*}
then $B_k(a)$ can be rewritten as
\begin{equation*}
  B_k(a)=
  \frac{s^{2+a}\Gamma(k+1+\sigma-\frac{sl}{2})}{\Gamma(k+1+\sigma)}
  \sum_{i=1}^k\frac{1}{(i+\sigma)^{2+a}}
  \frac{\Gamma(i+1+\sigma)}{\Gamma(i+1+\sigma-\frac{sl}{2})}.
\end{equation*}
Further, it follows from Lemma \ref{AM:lem:GammaRatio} that
\begin{equation*}
  \frac{\Gamma(k+1+\sigma-\frac{sl}{2})}{\Gamma(k+1+\sigma)}
  =(k+1+\sigma)^{-\frac{sl}{2}}
  +\mathcal{O}\left((k+1+\sigma)^{-\frac{sl}{2}-1}\right),
\end{equation*}
and it follows from \eqref{AM:eq:localOrder} and Lemma \ref{AM:lem:GammaRatio} that
\begin{align*}
  \sum_{i=1}^k\frac{1}{(i\!+\!\sigma)^{2+a}}
  \frac{\Gamma(i\!+\!1\!+\!\sigma)}
  {\Gamma(i\!+\!1\!+\!\sigma\!-\!\frac{sl}{2})}
  =&\sum_{i=1}^k
  \frac{(i\!+\!1\!+\!\sigma\!-\!\frac{sl}{2})^{\frac{sl}{2}}}
  {(i\!+\!\sigma)^{2+a}}+\mathcal{O}\left(\sum_{i=1}^k
  \frac{(i\!+\!1\!+\!\sigma\!-\!\frac{sl}{2})^{\frac{sl}{2}-1}}
  {(i+\sigma)^{2+a}}\right) \\
  =&\frac{C}{\frac{sl}{2}\!-\!1\!-\!a}
  (k\!+\!1\!+\!\sigma)^{\frac{sl}{2}-1-a}
  +\mathcal{O}\left((k\!+\!1\!+\!\sigma)^{\frac{sl}{2}-2-a}\right),
\end{align*}
where $C$ is a positive real number. Hence, we finally get
\begin{align*}
  B_k(a)=\frac{Cs^{2+a}}{\frac{sl}{2}\!-\!1\!-\!a}(k+1+\sigma)^{-1-a}
  +\mathcal{O}\left((k+1+\sigma)^{-2-a}\right),
\end{align*}
as desired.
\end{proof}

Combining Theorem \ref{AM:thm:ABm} and Lemma \ref{AM:lem:ABk}, we see that
\begin{lemma}\label{AM:lem:main1}
Suppose the conditions of Theorem \ref{AM:thm:ABm} hold. Then there are $C_A,C_B>0$ such that the bound
\begin{equation*}
  \mathbb{E}[F(x_{k+1})]-F_*\leqslant
  \frac{C_A\Gamma(1+\sigma)}{\Gamma(1+\sigma-\frac{sl}{2})}
  \frac{F(x_1)-F_*}{(k+1+\sigma)^{\frac{sl}{2}}}
  +\frac{C_Bs^2}{sl-2}\frac{LM_{d,p,1}^{(1)}}{k+1+\sigma}
\end{equation*}
holds in probability, where $sl>4$.
\end{lemma}

Note that for the $(k +1)$th iteration, the entire mean-variance analysis process is only related to the first $k$ iterations. Thus, combining \eqref{AM:eq:CRf} and Lemma \ref{AM:lem:main1}, we prove that
\begin{theorem}\label{AM:thm:main1}
Suppose the conditions of Theorem \ref{AM:thm:ABm} hold. Then for every $k\in\mathbb{N}$, there are $C_A,C_B>0$ such that the bound
\begin{equation*}
  \mathbb{E}[F(x_k)]-F_*\leqslant
  \frac{C_A\Gamma(1+\sigma)}{\Gamma(1+\sigma-\frac{sl}{2})}
  \frac{F(x_1)-F_*}{(k+\sigma)^{\frac{sl}{2}}}
  +\frac{C_Bs^2}{sl-2}\frac{LM_{d,p,1}^{(1)}}{k+\sigma}
\end{equation*}
holds in probability, where $sl>4$; that is, $\mathbb{E}[F(x_k)]-F_*=\mathcal{O}(1/k)$.
\end{theorem}

Similarly, combining Lemma \ref{AM:lem:ABk}, Theorems \ref{AM:thm:ABm2} and \ref{AM:thm:main1}, we see that
\begin{theorem}\label{AM:thm:main2}
Suppose the conditions of Theorem \ref{AM:thm:ABm2} hold. Then for every $k\in\mathbb{N}$, there are $C_A,C'_B>0$ such that the bound
\begin{equation*}
  \mathbb{E}[F(x_k)]-F_*\leqslant
  \frac{C_A\Gamma(1+\sigma)}{\Gamma(1+\sigma-\frac{sl}{2})}
  \frac{F(x_1)-F_*}{(k+\sigma)^{\frac{sl}{2}}}
  +\frac{C'_Bs^{2+\kappa'}}{sl-2-2\kappa'}
  \frac{LM_{d,p,2}^{(1)}}{(k+\sigma)^{1+\kappa'}}
\end{equation*}
holds in probability, where $\kappa'=\min(1,\kappa)$ and $sl>4$; that is, $\mathbb{E}[F(x_k)]-F_*=\mathcal{O}(1/k^{1+\kappa'})$.
\end{theorem}

Therefore, when $0<\kappa\leqslant1$, Assumption \ref{AM:ass:A1} and Theorem \ref{AM:thm:main2} implies that for every $1\leqslant j\leqslant k$, it holds that
\begin{equation*}
  \mathbb{E}[\|x_j-x_*\|_2^2]\leqslant2l^{-1}(\mathbb{E}[F(x_j)]-F_*)
  =\mathcal{O}(j^{-1-\kappa})=\mathcal{O}(j^{-1-\kappa}),
\end{equation*}
then we have
\begin{align*}
  \mathbb{E}[\|x_j-x_k\|_2^2]\leqslant&
  \mathbb{E}[\|x_j-x_*\|_2+\|x_k-x_*\|_2]^2 \\
  =&\mathbb{E}[\|x_j-x_*\|_2^2+2\|x_j-x_*\|_2\|x_k-x_*\|_2
  +\|x_k-x_*\|_2^2] \\
  \leqslant&2\mathbb{E}[\|x_j-x_*\|_2^2+\|x_k-x_*\|_2^2]
  =\mathcal{O}(j^{-1-\kappa}),
\end{align*}
together with Assumption \ref{AM:ass:A3}, we obtain
\begin{equation}\label{AM:eq:REC2}
  \|\mathbb{E}[x_j-x_k]\|_2^2=\mathcal{O}
  \Big(j^{-\kappa}\mathbb{E}[\|x_j-x_k\|_2^2]\Big)
  =\mathcal{O}\big(j^{-1-2\kappa}\big).
\end{equation}
Hence, from \eqref{AM:eq:delta2} and \eqref{AM:eq:REC2}, it is clear that $\|\mathbb{E}[x_j-x_k]\|_2^2=\mathcal{O}(j^{-1-\kappa})$ implies $\|\mathbb{E}[x_j-x_k]\|_2^2= \mathcal{O}(j^{-1-2\kappa})$ for every $0<\kappa\leqslant1$, which actually means
\begin{equation*}
  \|\mathbb{E}[x_j-x_k]\|_2^2=\mathcal{O}(j^{-2}),
\end{equation*}
that is, $\kappa'=1$; thus, we have the final result.
\begin{theorem}\label{AM:thm:final}
Suppose the conditions of Theorem \ref{AM:thm:ABm2} hold. Then for every $k\in\mathbb{N}$, there are $C_A,C'_B>0$ such that the bound
\begin{equation*}
  \mathbb{E}[F(x_k)]-F_*\leqslant
  \frac{C_A\Gamma(1+\sigma)}{\Gamma(1+\sigma-\frac{sl}{2})}
  \frac{F(x_1)-F_*}{(k+\sigma)^{\frac{sl}{2}}}+
  \frac{C'_Bs^3}{sl-4}\frac{LM_{d,p,2}^{(1)}}{(k+\sigma)^2}
\end{equation*}
holds in probability, where $sl>4$; that is, $\mathbb{E}[F(x_k)]-F_*=\mathcal{O}(1/k^2)$.
\end{theorem}

\section{Conclusions}
\label{AM:s5}

In this work, we discussed the question of whether it is possible to apply a gradient averaging strategy to improve on the sublinear convergence rates without any increase in storage for SG methods. We proposed a gradient averaging strategy and proved that under the variance dominant condition, the proposed strategy could improve the convergence rate $\mathcal{O}(1/k)$ to $\mathcal{O}(1/k^2)$ in probability without any increase in storage for the strongly convex objectives with Lipschitz gradients. This result suggests how we should control the stochastic gradient iterations to improve the rate of convergence in practice.

\vskip 0.2in
%\bibliography{sample}
\bibliographystyle{siamplain}
\bibliography{MReferences}

\end{document}